\documentclass[a4paper, 11pt]{amsart}
\usepackage{amsmath}
\usepackage{amssymb}
\usepackage{cases}
\usepackage{enumitem}
\allowdisplaybreaks[4]

\newtheorem{thm}{Theorem}[section]
\newtheorem{prop}[thm]{Proposition}

\newtheorem{lem}[thm]{Lemma}

\newtheorem{cor}[thm]{Corollary}

\newtheorem{claim}[thm]{Claim}

\theoremstyle{definition}
\newtheorem{definition}[thm]{Definition}
\newtheorem{example}[thm]{Example}

\theoremstyle{remark}
\newtheorem{remark}[thm]{Remark}

\numberwithin{equation}{section}

\newcommand{\bQ}{\mathbb{Q}}
\newcommand{\bR}{\mathbb{R}}

\newcommand{\bP}{\mathbb{P}}

\newcommand\OO{{\mathcal{O}}}

\newcommand{\mult}{\operatorname{mult}}

\newcommand{\Vol}{\operatorname{Vol}}

\usepackage{todonotes}

\begin{document}

\title{An effective upper bound for anti-canonical volumes of singular Fano threefolds}
\date{\today}
\author{Chen Jiang}
\address{Chen Jiang, Shanghai Center for Mathematical Sciences, Fudan University, Jiangwan Campus, Shanghai, 200438, China}
\email{chenjiang@fudan.edu.cn}

\author{Yu Zou}
\address{Yu Zou, Yau Mathematical Sciences Center, Tsinghua University, Beijing, 100084, China}
\email{fishlinazy@tsinghua.edu.cn}


\begin{abstract}
For a real number $0<\epsilon<1/3$, we show that the anti-canonical volume of an
$\epsilon$-klt Fano $3$-fold is at most $3200/\epsilon^4$ and the order $O(1/\epsilon^4)$ is sharp.
\end{abstract}

\maketitle
\pagestyle{myheadings} \markboth{\hfill C.~Jiang \& Y.~Zou
\hfill}{\hfill An effective upper bound for anti-canonical volumes of singular Fano $3$-folds\hfill}

\tableofcontents

\section{Introduction}
Throughout this paper, we work over the field of complex numbers $\mathbb{C}$.

A normal projective variety $X$ is a {\it Fano} variety if $-K_X$ is ample. According to the minimal model program, Fano varieties form a fundamental class in the birational classification of algebraic varieties. 

One recent breakthrough in birational geometry is the proof of the Borisov--Alexeev--Borisov conjecture by Birkar \cite{Bir19, Bir21}, which states that for a fixed positive integer $d$ and a positive real number $\epsilon$, the set of $d$-dimensional Fano varieties with $\epsilon$-klt singularities forms a bounded family. During the proof, one important step is to establish the upper bound for the anti-canonical volume $(-K_X)^d$ for an $\epsilon$-klt Fano variety $X$ of dimension $d$ (\cite[Theorem~1.6]{Bir19}).

Motivated by the classification theory of $3$-folds, we mainly focus on the anti-canonical volume $(-K_X)^3$ for an $\epsilon$-klt Fano $3$-fold $X$.
In this direction, 
Lai \cite{Lai16} gave an upper bound for those $X$ which are $\mathbb{Q}$-factorial and of Picard rank $1$, which is over  $O((\frac{4}{\epsilon})^{384/\epsilon^5})$; later, the first author \cite{Jiang21} showed the existence of a non-explicit upper bound; recently, Birkar \cite{Bir22} gave the first explicit upper bound, which is about $O(\frac{2^{1536/\epsilon^3}}{\epsilon^9})$.

The main goal of this paper is to provide a reasonably small explicit upper bound with a sharp order, for the anti-canonical volume of an $\epsilon$-klt Fano $3$-fold.
Here we state the result for a larger class of varieties containing $\epsilon$-klt Fano $3$-folds. Recall that a normal projective variety $X$ is said to be {\it of $\epsilon$-Fano type} if there exists an effective $\mathbb{R}$-divisor $B$ such that $(X, B)$ is $\epsilon$-klt and $-(K_X+B)$ is ample.

\begin{thm}\label{mainthm}
Fix a real number $0<\epsilon<\frac{1}{3}$. Let $X$ be a $3$-fold of $\epsilon$-Fano type. Then $$\Vol(X, -K_{X})< \frac{3200}{\epsilon^4}.$$
\end{thm}

The following example shows that the
order $O(\frac{1}{\epsilon^4})$ in Theorem~\ref{mainthm} is sharp.

\begin{example}
 Ambro \cite[Example~6.3]{Ambro} showed that for each positive integer $q$, there exists a projective toric $3$-fold pair $(X, B)$ such that
\begin{itemize}
\item $(X, B)$ is $\frac{1}{q}$-lc,
 \item $-(K_X+B)$ is ample (in fact, $-K_X$ is ample as $\rho(X)=1$), and
 \item $(-K_X)^3> (-(K_X+B))^3=\frac{u_{4, q}}{q^4}=O(q^4)$.
\end{itemize} 
\end{example}

\begin{remark}\label{rem:compareJiang21}
 The ideas of this paper originate from \cite{Jiang18, Jiang21}, which we briefly explain in the following. Given a $3$-fold $X$ of $\epsilon$-Fano type, in order to give an upper bound of the anti-canonical volume of $X$, we may reduce to the case that $X$ admits a Mori fiber structure $X\to T$. Then we can split the discussion according to $\dim T\in \{0,1,2\}$. The case $\dim T=0$ was solved in \cite{Prok07} (or \cite{Lai16}) and other 2 cases were solved in 
 \cite{Jiang21}. The main obstructions of getting a reasonably small explicit upper bound in \cite{Jiang21} are the following two issues:
 \begin{enumerate}
 \item when $\dim T=1$, we reduce the upper bound problem to finding a lower bound of certain log canonical thresholds on surfaces, called $\mu(2, \epsilon)$, but the lower bound in \cite{Jiang21} is extremely small so that the resulting upper bound is extremely large;
 
\item when $\dim T=2$, we reduce the upper bound problem to the boundedness of such surfaces $T$ (more precisely, the existence of very ample divisors on $T$ with bounded self-intersection numbers), but the geometry of those $T$ is quite complicated which makes the upper bound non-explicit; in fact, even if we can classify those surfaces $T$, then the resulting upper bound will be explicitly computable but still extremely large.
 \end{enumerate}

This paper is devoted to solving these two issues and we put two main ingredients into the recipe. 

The first one is a new reduction originate from \cite{JZ21} (Proposition~\ref{prop MFS}), which
shows that we may further assume that $X$ admits a better fibration structure $X\to S$ so that if $\dim S=2$, then there exists a free divisor on $S$ with small self-intersection number. This solves the second issue.

The second one is a more detailed estimate on the lower bound $\mu(2, \epsilon)$ (Theorem~\ref{thm lct1}), which solves the first issue. We significantly improve the lower bound from about $O(2^{-64/\epsilon^3})$ in \cite{Jiang21} to $O(\epsilon^3)$ (Theorem~\ref{thm lct1}).
\end{remark}

\begin{remark}\label{rem:compareBir22}
We shall also compare our result with \cite{Bir22}. 
The method used in \cite[Theorem~1.2]{Bir22} is a slight modification of \cite[Theorem~1.6]{Bir19}, which is different from our method (but it is a more general strategy that works in any dimension). One thing we share in common is that, \cite[Theorem~1.2]{Bir22} also reduces the problem to finding some kind of lower bound of log canonical thresholds on surfaces (\cite[Lemma~2.2]{Bir22}).
Here note that the constant $\mu(2, \epsilon)$ in \cite[Lemma~2.2]{Bir22} is 
$\frac{1}{\mu(2, \epsilon)}$ in our terminology.
If we replace \cite[Lemma~2.2]{Bir22} by Theorem~\ref{thm lct1} in the proof of \cite[Theorem~1.2]{Bir22}, then we get an explicit upper bound about $O(\frac{1}{\epsilon^{12}})$. 
\end{remark}

\section{Preliminaries}\label{sec 2}
We adopt standard notation and definitions in \cite{KM} and will freely use them.
We use $\sim_\bQ, \sim_\bR, \equiv$ to denote $\bQ$-linear equivalence, $\bR$-linear equivalence, and numerical equivalence respectively.

\subsection{Singularities of pairs}
\begin{definition}
A {\it pair} $(X, B)$ consists of a normal variety $X$ and an effective
$\bR$-divisor $B$ on $X$ such that
$K_X+B$ is $\bR$-Cartier.
\end{definition}

\begin{definition}\label{def sing}
Let $(X, B)$ be a pair. Let $f: Y\to X$ be a log
resolution of $(X, B)$, write
$$
K_Y =f^*(K_X+B)+\sum a_iE_i,
$$
where $E_i$ are distinct prime divisors on $Y$ satisfying $f_*(\sum a_iE_i)=-B$. The number $a_i+1$ is called the {\it log discrepancy} of $E_i$ with respect to $(X, B)$, and is denoted by $a(E_i, X, B)$. 
The pair $(X,B)$ is called
\begin{enumerate} 
\item \emph{Kawamata log terminal} ({\it klt},
for short) if $a_i+1>0$ for all $i$;

\item \emph{log canonical} (\emph{lc}, for
short) if $a_i+1\geq 0$ for all $i$;

\item \emph{$\epsilon$-klt}, if $a_i+1> \epsilon$ for all $i$, for some $0<\epsilon<1$;

\item \emph{$\epsilon$-lc}, if $a_i+1\geq \epsilon$ for all $i$, for some $0<\epsilon<1$;

\item \emph{terminal} if $a_i> 0$ for all $f$-exceptional divisors $E_i$ and for all $f$.

\end{enumerate}
Usually, we write $X$ instead of $(X,0)$ in the case when $B=0$.

\end{definition}

\subsection{Varieties of Fano type}

\begin{definition}
 A variety $X$ is said to be 
 {\it of $\epsilon$-Fano type} 
 if $X$ is projective and there exists an effective $\mathbb{R}$-divisor $B$ such that $-(K_X+B)$ is ample and $(X, B)$ is 
 $\epsilon$-klt for some $0<\epsilon<1$.
\end{definition}


\subsection{Volumes}
\begin{definition}
Let $X$ be a normal projective variety of dimension $n$ and let $D$ be a Cartier divisor on $X$. The {\it volume} of $D$ is defined by
$$
\Vol(X, D)=\limsup_{m\to \infty}\frac{h^0(X, \OO_X(mD))}{m^n/n!}.
$$
Moreover, by homogeneous property of volumes, the definition can be extended to $\bQ$-Cartier $\bQ$-divisors. Note that if $D$ is a nef $\bQ$-Cartier $\bQ$-divisor, then $\Vol(X, D)=D^n$. We refer to \cite[2.2.C]{Positivity1} for more details and properties on volumes of divisors. 
\end{definition}

\section{A lower bound of log canonical thresholds on surfaces}\label{sec 3}

The main goal of this section is to prove the following theorem on certain log canonical thresholds on surfaces. This is the main ingredient of this paper and the most technical part.

\begin{thm}\label{thm lct1}
Fix $0<\epsilon<\frac{1}{3}$.
Let $S$ be a smooth projective surface. Suppose that there exists a real number $0<t<1$ and effective $\bR$-divisors $B, D$ on $S$ such that 
\begin{itemize}
 \item $(S, B)$ is $\epsilon$-lc;
 \item $(S, (1-t)B+tD)$ is not klt;
 \item $B\sim_\bR D \sim_\bR -K_S$ and $-K_S$ is big. 
\end{itemize}
Then $t>\frac{3\epsilon^3}{400}.$
\end{thm}

As an immediate corollary, we confirm the generalized Ambro's conjecture (\cite[Conjecture~2.7]{Jiang21}) in dimension $2$ with a greatly improved lower bound.
\begin{cor}[{cf. \cite[Theorem~2.8]{Jiang21}}]\label{cor:genambro}
The generalized Ambro's conjecture (\cite[Conjecture~2.7]{Jiang21}) holds in dimension $2$ with $\mu(2, \epsilon)>\frac{3\epsilon^3}{400}$.
\end{cor}

The reduction from Corollary~\ref{cor:genambro} to Theorem~\ref{thm lct1} is standard (see \cite[\S5, Page~1583]{Jiang21}).

\begin{remark}\label{rem:ambroex}
While the constant term might be improved slightly, the order $\epsilon^3$ in Theorem~\ref{thm lct1} is sharp.
In fact, Ambro \cite[Theorem~1.1, Example~6.3]{Ambro} showed that for each positive integer $q$, there exists a projective toric surface pair $(X, \Delta)$ such that
\begin{itemize}
\item $(X, \Delta)$ is $\frac{1}{q}$-lc,
 \item $-(K_X+\Delta)$ is ample, and
 \item there exists an effective $\mathbb{Q}$-divisor $H\sim_\mathbb{Q} -(K_X+\Delta)$ such that $(X, \Delta+tH)$ is not klt for $t=\frac{1}{(q+1)(q^2+q+1)}=O(\frac{1}{q^3})$.
\end{itemize} 
We can modify this example to satisfy assumptions in Theorem~\ref{thm lct1}.
Take $A\sim_\mathbb{Q} -(K_X+\Delta)$ to be a sufficiently general ample effective $\mathbb{Q}$-divisor such that $(X, \Delta+A)$ is still $\frac{1}{q}$-lc. 
Take $\pi: S\to X$ to be the minimal resolution of $X$, then $-K_S$ is big as $-K_X$ is big. We may write 
$$
K_S+B=\pi^*(K_X+\Delta+A)\sim_\mathbb{Q}0
$$
for some effective $\mathbb{Q}$-divisor $B\geq \pi^*A$. 
Then $(S, B)$ is $\frac{1}{q}$-lc.
In this case, $D:=B-\pi^*A+\pi^*H$ is an effective $\mathbb{Q}$-divisor such that $D\sim_\mathbb{Q}B\sim_\mathbb{Q}-K_S$
and the pair 
$(S, (1-t)B+tD)$ is not klt
as 
$$
K_S+(1-t)B+tD=\pi^*(K_X+\Delta+tH+ (1-t)A).$$
\end{remark}

\subsection{Weighted dual graphs}
In this subsection, we recall basic knowledge of weighted dual graphs of resolutions of surface singularities from \cite{Kol92} or \cite[4.1]{KM}. 

Let $Y$ be a normal surface and let $\pi: Y'\to Y$ be a resolution with $\pi$-exceptional curves $\{E_i\}_i$. The {\em weighted dual graph} $\Gamma$ of $\pi$ is defined as the following: each vertex $v_i$ of $\Gamma$ corresponds to a $\pi$-exceptional curve $E_i$, and it has a positive weight $-E_i^2$; two vertices $v_i$ and $v_j$ are connected by an edge of weight $m=(E_i\cdot E_j)$ if $(E_i\cdot E_j)\neq 0$. 

If $Y$ is klt and has a unique singular point, then $\Gamma$ is a tree with simple edges and all $\pi$-exceptional curves are smooth rational curves by \cite[3.2.7~Lemma]{Kol92} or \cite[Theorem~4.7]{KM}. 
In this case, denote by $\overline{v_iv_j}$ the path from $v_i$ to $v_j$, i.e., the unique shortest chain in $\Gamma$ joining $v_i$ and $v_j$. 

For any subgraph $\Gamma'\subset \Gamma$, define $\Delta(\Gamma')$ to be the absolute value of the determinant of the matrix $[(E_{i}\cdot E_{j})]$, made up by vertices in $\Gamma'$. Here $\Delta(\emptyset)=1$ by default.

We will often use the following lemma to compute log discrepancies and multiplicities of exceptional divisors.

 \begin{lem}[{cf. \cite[(3.1.10)]{Kol92}}]\label{lem:LD}
Let $\pi: Y'\to Y$ be a resolution of a klt surface singularity $P\in Y$. 
Suppose that the set of $\pi$-exceptional curves is $\{E_1, E_2, \dots, E_n\}$. 
Denote by $\Gamma$ the weighted dual graph of $\pi$. \begin{enumerate}
 \item Then for each $1\leq k\leq n$, the log discrepancy
$$a(E_k, Y, 0)=\frac{\sum_{j=1}^n (2-(\sum_{i\neq j}E_{i} \cdot E_j) )\cdot\Delta(\Gamma\setminus \overline{v_kv_j})}{\Delta(\Gamma)}.$$

\item 
If $C$ is an irreducible curve on $Y$, then for each $1\leq k\leq n$, 
$$\mult_{E_k} \pi^*C=\frac{\sum_{j=1}^n(\pi^{-1}_*C\cdot E_j)\cdot\Delta(\Gamma\setminus \overline{v_kv_j})}{\Delta(\Gamma)}.$$
\end{enumerate}

\end{lem}
\begin{proof}
(1) is just \cite[(3.1.10)]{Kol92} and (2) can be deduced in the same way by applying \cite[3.1.9~Lemma]{Kol92} and the Cramer's rule.
\end{proof}

The following lemma will be used to deal with the weighted dual graph of the minimal resolution of a cyclic quotient singularity.

 \begin{lem}\label{lem:chain}
 Let $\Gamma$ be a chain with vertices $v_1, \dots, v_n$ ordering in the natural sense that $v_i$ is connected to $v_{i+1}$ by an edge for $1\leq i\leq n-1$. Suppose that for each $1\leq i\leq n$, the weight of $v_i$ is $m_i$ with $m_i\geq 2$.
 Then the following assertions hold:
 \begin{enumerate}
 \item $\Delta(\Gamma)=m_1\cdot \Delta(\Gamma\setminus\{v_1\})-\Delta(\Gamma\setminus\overline{v_1v_2})$;
 
 \item $\Delta(\Gamma\setminus\overline{v_1v_k})=m_{k+1}\cdot \Delta(\Gamma\setminus\overline{v_1v_{k+1}})-\Delta(\Gamma\setminus\overline{v_1v_{k+2}})$ for $1\leq k\leq n-2$;

\item $\Delta(\Gamma)>\Delta(\Gamma\setminus\{v_1\})>\Delta(\Gamma\setminus\overline{v_1v_2})>\dots> \Delta(\Gamma\setminus\overline{v_1v_n})=1$;
 
 \item $\Delta(\Gamma)\geq n+1$ and $\Delta(\Gamma\setminus\overline{v_1v_k})\geq n-k+1$ for $1\leq k\leq n$; moreover,
 the equalities hold if all $m_i=2$; 
 
 \item if $\Delta(\Gamma)=\Delta(\Gamma\setminus\{v_1\})+1$, then $\Delta(\Gamma)=n+1$;


 \item if $m_{i_0}\geq3$ for some $1\leq i_0\leq n$, then $\Delta(\Gamma)>(i_{0}+1)\Delta(\Gamma\setminus\overline{v_{1}v_{i_0}})$.
 
 \end{enumerate}
\end{lem}
\begin{proof} 
Assertions (1) and (2) can be calculated easily from determinants.
By Assertions (1)(2) and the fact that $m_i\geq 2$, 
we have 
\begin{align}\label{eq:Delta1}
 \Delta(\Gamma)\geq 2 \Delta(\Gamma\setminus\{v_1\})-\Delta(\Gamma\setminus\overline{v_1v_2})
\end{align} 
and 
\begin{align}\label{eq:Delta2}
\Delta(\Gamma\setminus\overline{v_1v_k})\geq 2 \Delta(\Gamma\setminus\overline{v_1v_{k+1}})-\Delta(\Gamma\setminus\overline{v_1v_{k+2}})\end{align} 
for $1\leq k\leq n-2$.
So Assertion (3) follows inductively from the fact that 
$$
\Delta(\Gamma\setminus\overline{v_1v_{n-1}})-\Delta(\Gamma\setminus\overline{v_1v_{n}})=m_n-1>0.
$$
Assertion (4) follows from Assertion (3) and direct computation. 

For Assertion (5), the assumption combining with Eq.~\eqref{eq:Delta1} and Eq.~\eqref{eq:Delta2} implies that 
$$
\Delta(\Gamma\setminus\overline{v_1v_{k+1}})-\Delta(\Gamma\setminus\overline{v_1v_{k+2}})=1
$$
for $1\leq k\leq n-2$. Moreover, Eq.~\eqref{eq:Delta1} and Eq.~\eqref{eq:Delta2} become equalities, and hence all $m_i=2$. So $\Delta(\Gamma)=n+1$ by Assertion~(4).

By applying 
Eq.~\eqref{eq:Delta1} and Eq.~\eqref{eq:Delta2} inductively, 
one can see that 
\begin{align*}
 \Delta(\Gamma)\geq j \Delta(\Gamma\setminus\overline{v_1v_{j-1}})-(j-1)\Delta(\Gamma\setminus\overline{v_1v_j})
\end{align*}
for $1\leq j\leq n$. Here we set $\Gamma\setminus\overline{v_1v_{0}}=\Gamma$.
Then by Assertions (1)(2)(3), for $1\leq j\leq n$, we have
\begin{align*}
 \Delta(\Gamma){}&\geq j(m_j \Delta(\Gamma\setminus\overline{v_1v_{j}})-\Delta(\Gamma\setminus\overline{v_1v_{j+1}}))-(j-1)\Delta(\Gamma\setminus\overline{v_1v_j})\\
 {}&=(j m_j-j+1) \Delta(\Gamma\setminus\overline{v_1v_{j}})-j \Delta(\Gamma\setminus\overline{v_1v_{j+1}})\\
 {}&>(j m_j-2j+1) \Delta(\Gamma\setminus\overline{v_1v_{j}}).
\end{align*}
Here we set $\Delta(\Gamma\setminus\overline{v_1v_{n+1}})=0$.
In particular, if $m_{i_0}\geq 3$, then $\Delta(\Gamma)>(i_{0}+1)\Delta(\Gamma\setminus\overline{v_{1}v_{i_0}})$.
\end{proof}

\subsection{Geometric structure of $\delta$-lc surface pairs}

\begin{lem}\label{lem:multEC}
Fix a real number $0<\delta<\frac{1}{6}$ and a positive integer $N$.
Let $Y$ be a normal surface and let $C$ be an irreducible curve on $Y$ such that $(Y, (1-\delta)C)$ is $\delta$-lc.
Let $\pi: Y'\to Y$ be the minimal resolution of $Y$.
Suppose that $\rho(Y'/Y)\leq N$. Then $\mult_{E}\pi^*C> \frac{\delta}{N+1}$ for any $\pi$-exceptional divisor $E$ on $Y'$ such that $\pi(E)\in C$.

\end{lem}

\begin{proof}
By shrinking $Y$ if necessary, we may assume that $P\in Y$ is the only singular point on $Y$ and $P\in C$. Denote by $\Gamma$ the weighted dual graph of $\pi$. For any $\pi$-exceptional divisor $E$, clearly $\mult_{E}\pi^*C$ is a positive rational number and its denominator divides $\Delta(\Gamma)$ by Lemma~\ref{lem:LD}. So it suffices to show that $\Delta(\Gamma)<\frac{N+1}{\delta}$.

By \cite[Corollary~6.0.9]{Prok01}, $(Y, C)$ is lc.
Then the weighted dual graph $\Gamma$ of the minimal resolution $\pi: Y'\to Y$ are classified into $3$ cases as in \cite[Theorem~4.15]{KM}. 
We split the discussion into these $3$ cases. 
Denote by $m=\rho(Y'/Y)$ the number of $\pi$-exceptional curves on $Y'$.

As $(Y, (1-\delta)C)$ is $\delta$-lc, $a(E, Y, (1-\delta)C)\geq \delta$ for any prime divisor $E$ over $Y$. We will apply this fact to some specially chosen $E$.

\medskip

{\bf Case (1)}: For the case \cite[Theorem~4.15(1)]{KM}, $\Gamma$ is a chain with vertices $v_1, \cdots, v_m$ corresponding to $\pi$-exceptional curves $E_1, \cdots, E_m$ such that $\pi^{-1}_*C$ intersects $E_1$ and $E_m$. 

If $m\geq 2$, by Lemma~\ref{lem:LD}, we have
\[
 a(E_1, Y, 0)= \mult_{E_1}\pi^*C=\frac{\Delta(\Gamma\setminus\{v_1\})+\Delta(\emptyset)}{\Delta(\Gamma)}.
\]
Then 
\begin{align*}
 a(E_1, Y, (1-\delta)C)={}&a(E_1, Y, 0)-(1-\delta)\mult_{E_1}\pi^*C\\
 ={}&\delta\cdot\frac{\Delta(\Gamma\setminus\{v_1\})+1}{\Delta(\Gamma)}.
\end{align*}
Therefore, $a(E_1, Y, (1-\delta)C)\geq \delta$ implies that $\Delta(\Gamma\setminus\{v_1\})+1\geq \Delta(\Gamma)$.
So by Lemma~\ref{lem:chain}(5), $\Delta(\Gamma)=m+1\leq N+1$.

If $m=1$, then a similar computation by Lemma~\ref{lem:LD} shows that 
\[
 a(E_1, Y, (1-\delta)C)=\delta\cdot\frac{2\Delta(\Gamma\setminus\{v_1\})}{\Delta(\Gamma)}=\frac{2\delta}{\Delta(\Gamma)}.
\]
Therefore, $a(E_1, Y, (1-\delta)C)\geq \delta$ implies that $\Delta(\Gamma)\leq 2$.

\medskip

{\bf Case (2)}: For the case \cite[Theorem~4.15(2)]{KM}, we have $m\geq 3$. 

If $m\geq 4$, then
$\Gamma$ is a tree with only one fork. Let $E_{\text{f}}$ be the $\pi$-exceptional curve corresponding to the fork vertex $v_{\text{f}}$.
Then $\Gamma\setminus \{v_{\text{f}}\}=\Gamma'\cup\{v_1\}\cup\{v_2\}$, where $v_1, v_2$ correspond to $(-2)$-curves intersecting $E_{\text{f}}$ and $\Gamma'$ is the chain corresponding to curves connecting $E_{\text{f}}$ and $\pi^{-1}_*C$.

Then by Lemma~\ref{lem:LD},
\begin{align*}
 {}&a(E_{\text{f}}, Y, 0)\\
 ={}&\frac{\Delta(v_1)\cdot \Delta(v_2)-\Delta(\Gamma')\cdot \Delta(v_1)\cdot \Delta(v_2)+\Delta(\Gamma')\cdot \Delta(v_2)+\Delta(\Gamma')\cdot \Delta(v_1)}{\Delta(\Gamma)}\\
 ={}&\frac{4}{\Delta(\Gamma)}
\end{align*}
and
 \[ \mult_{E_{\text{f}}}\pi^*C=\frac{\Delta(v_1)\cdot \Delta(v_2)}{\Delta(\Gamma)}=\frac{4}{\Delta(\Gamma)}.
\]
Then $a(E_{\text{f}}, Y, (1-\delta)C)=\frac{4\delta}{\Delta(\Gamma)} $. Therefore, $a(E_{\text{f}}, Y, (1-\delta)C)\geq \delta$ implies that $\Delta(\Gamma)\leq4$.

If $m=3$, then
$\Gamma$ is a chain consisting of $3$ vertices $v_1, v_2, v_3$ corresponding to $E_1, E_2, E_3$ such that $E_1$ and $E_3$ are $(-2)$-curves and $\pi^{-1}_*C$ intersects $E_2$.
Then a similar computation by Lemma~\ref{lem:LD} shows that 
$a(E_{2}, Y, (1-\delta)C)=\frac{4\delta}{\Delta(\Gamma)}$. Therefore, $a(E_{2}, Y, (1-\delta)C)\geq \delta$ implies that $\Delta(\Gamma)\leq4$.

\medskip

{\bf Case (3)}:
For the case \cite[Theorem~4.15(3)]{KM}, $\Gamma$ is a chain with vertices $v_1, \cdots, v_m$ corresponding to $\pi$-exceptional curves $E_1, \cdots, E_m$ such that $\pi^{-1}_*C$ intersects $E_1$. If all $E_i$ are $(-2)$-curves for $1\leq i\leq m$, then by Lemma~\ref{lem:chain}(4), $\Delta(\Gamma)=m+1\leq N+1$. 
If 
$-E_{i_0}^2\geq3$ for some $1\leq i_0\leq m$, take $i_0$ to be the minimal one,
then by Lemma~\ref{lem:LD} and Lemma~\ref{lem:chain}(4),
\begin{align*}
a(E_{i_0}, Y, (1-\delta)C)
 ={}&\frac{\Delta(\Gamma\setminus\overline{v_{i_0}v_m})}{\Delta(\Gamma)}+\delta\cdot\frac{\Delta(\Gamma\setminus\overline{v_{1}v_{i_0}})}{\Delta(\Gamma)} \\
 ={}&\frac{i_0}{\Delta(\Gamma)}+\delta\cdot\frac{\Delta(\Gamma\setminus\overline{v_{1}v_{i_0}})}{\Delta(\Gamma)} \\
 <{}&\frac{i_0}{\Delta(\Gamma)}+\frac{\delta}{i_0+1}.
\end{align*}
Here the last inequality is by Lemma~\ref{lem:chain}(6).
Therefore, $a(E_{i_0}, Y, (1-\delta)C)\geq \delta$ implies that $\Delta(\Gamma)<\frac{i_0+1}{\delta}\leq\frac{N+1}{\delta}$. 
\end{proof}

\begin{lem}[{cf. \cite[Claim~2]{Jiang21}}]\label{lem:C2>-2}
Fix $0<\delta<1$. 
Let $X$ be a smooth projective surface and let $(X, B)$ be a $\delta$-lc pair such that $K_X+B\equiv 0$.
Then $C^2\geq -\frac{2}{\delta}$ for any irreducible curve $C$ on $X$.
\end{lem}

\begin{proof}
We may assume that $C^2<0$.
Then by the genus formula,
\begin{align*}
 -2{}&\leq 2p_a(C)-2=(K_X+C)\cdot C\\
 {}&= \delta C^2+(K_X+(1-\delta)C)\cdot C\\
 {}&\leq \delta C^2+(K_X+B)\cdot C=\delta C^2.
\end{align*}
\end{proof}

\begin{lem}\label{lem:KC<3}
Let $X$ be a normal projective surface such that $K_X$ is $\mathbb{Q}$-Cartier and not pseudo-effective. Then either $X\simeq \mathbb{P}^2$ or $X$ is covered by a family of rational curves $C$ such that $(-K_X\cdot C)\leq 2$.
\end{lem}
\begin{proof}
Take $\pi: X'\to X$ to be the minimal resolution of $X$, then $K_{X'}+G=\pi^*K_X$ where $G$ is an effective $\mathbb{Q}$-divisor. 
Then $K_{X'}$ is not pseudo-effective as $K_X$ is not pseudo-effective. 
Suppose that $X\not \simeq \mathbb{P}^2$, then clearly $X'\not \simeq \mathbb{P}^2$. By the standard minimal model program, there exists a morphism $X'\to T$ whose general fibers are $\mathbb{P}^1$. Therefore $X'$ is covered by a family of rational curves $C'$ such that $(-K_{X'}\cdot C')=2$.
So $(-K_X\cdot \pi(C'))\leq (-K_{X'}\cdot C')= 2$.
\end{proof}

\subsection{Proof of Theorem~\ref{thm lct1}}
In this subsection we give the proof of Theorem~\ref{thm lct1}.

By \cite[Lemma~3.1]{Jiang13}, there is a birational morphism $g: S\to S'$ where $S'$ is $\bP^2$ or the $n$-th Hirzebruch surface $\mathbb{F}_n$ with $n\leq\frac{2}{\epsilon}$. Since $B\sim_{\bR}D\sim_{\bR}-K_S$, we may write 
\begin{align*}
 K_S+B={}&g^*(K_{S'}+g_*B);\\
 K_S+(1-t)B+tD={}&g^*(K_{S'}+(1-t)g_*B+tg_*D).
\end{align*}
Hence $({S'}, g_*B)$ is $\epsilon$-lc and $({S'}, (1-t)g_*B+tg_*D)$ is not klt. 
By replacing the triple $(S, B, D)$ with $(S', g_*B, g_*D)$, we may assume that $S$ is $\bP^2$ or $\mathbb{F}_n$ for some $n\leq\frac{2}{\epsilon}$.

Fix a positive real number $0<\delta < {\epsilon}$ such that $\delta<\frac{1}{6}$. Take $$
t_0=\max\{s\in \mathbb{R}\mid (S, (1-s)B+s D) \text{ is } \delta\text{-lc}\}.
$$
Then clearly $0<t_0<t$. In the following, we will show that
\begin{align}
 t_0>\frac{\delta^2(\epsilon-\delta)}{16+4\delta+\delta^2(\epsilon-1)}.\label{eq:t0>delta}
\end{align}
In particular, we can take $\delta=\frac{\epsilon}{2}$, then Eq.~\eqref{eq:t0>delta} implies that $t>t_0>\frac{3\epsilon^3}{400}.$

By the definition of $t_0$, there exists a prime divisor $E$ over $S$ such that
$$a(E, S, (1-t_0)B+t_0 D)=\delta.$$
If $E$ is a prime divisor on $S$, then as $(S, B)$ is $\epsilon$-lc, we have
\begin{align*}
 {}&\mult_E B\leq 1-\epsilon;\\
 {}&\mult_E ((1-t_0)B+t_0 D)= 1-\delta.
\end{align*}
So \[t_0\geq\frac{\epsilon-\delta}{\mult_{E}D-1+\epsilon}\geq\frac{\epsilon(\epsilon-\delta)}{2+3\epsilon+\epsilon^2}\]
which implies Eq.~\eqref{eq:t0>delta}. Here we used the fact that \[\mult_{E}D\leq \begin{cases}
3 & \text{if } S=\mathbb{P}^2;\\
n+4\leq \frac{2}{\epsilon}+4 & \text{if } S=\mathbb{F}_n,
\end{cases} 
\]
 by applying \cite[Lemma~3.3]{Jiang21} to a general point on $E$.

So from now on we may assume that $E$ is exceptional over $S$.
By \cite[Corollary~1.4.3]{BCHM},
there exists a projective birational morphism $f: Y\to S$ such that $E$ is the unique $\pi$-exceptional divisor on $Y$. 
We have 
\begin{align}
 K_Y+(1-t_0)B_Y+t_0D_Y+(1-\delta)E= f^*(K_{S}+(1-t_0)B+t_0D).\label{eq:Y=f*S}
\end{align}
Here $B_Y$ and $D_Y$ are strict transforms of $B$ and $D$ on $Y$. 
Write
\begin{align*}
K_{Y}+B_Y+bE={}&f^*(K_S+B)\equiv 0, \\
K_{Y}+D_Y+dE={}&f^*(K_S+D)\equiv 0.
\end{align*}
Then $(1-t_0)b+t_{0}d=1-\delta$. Since $b\leq 1-\epsilon$ as $(S, B)$ is $\epsilon$-lc, we have
\begin{align}t_0\geq\frac{\epsilon-\delta}{d+\epsilon-1}.\label{eq:t0>delta-2}
\end{align}
So in order to bound $t_0$ from below, we need to bound $d$ from above.

By Eq.~\eqref{eq:Y=f*S}, we know that $(Y, (1-t_0)B_Y+t_0D_Y+(1-\delta)E)$ is $\delta$-lc and
\begin{align*}
 K_Y+(1-t_0)B_Y+t_0D_Y+(1-\delta)E\equiv 0.
\end{align*}
We can run a $(K_Y+(1-t_0)B_Y+t_0 D_Y)$-MMP (which is also a $(-E)$-MMP) on $Y$ to get a Mori fiber space $Y'\to Z$ such that $E'$ is ample over $Z$,  where $E'$ is the strict transform of $E$ on $Y'$.
Here $(Y', (1-\delta)E')$ is again $\delta$-lc by the negativity lemma.
We have $-K_{Y'}\equiv D_{Y'}+dE'$, where $D_{Y'}$ is the strict transform of $D_Y$ on $Y'$.

If $\dim Z=1$, then a general fiber of $Y'\to Z$ is a smooth rational curve. By restricting $-K_{Y'}\equiv D_{Y'}+dE'$ on a general fiber, we get $d\leq2$.

If $\dim Z=0$, then $Y'$ is of Picard rank $1$. In this case, $-K_{Y'}\equiv e E'$ for some $e\geq d$. 
If $Y'\simeq \mathbb{P}^2$ then clearly $d\leq e\leq 3$.
So we may assume that $Y'\not\simeq \mathbb{P}^2$.
By Lemma~\ref{lem:KC<3}, there is a general rational curve $C$ such that $(-K_{Y'}\cdot C)\leq 2$. 
Take $\pi': Y'_{\min}\to Y'$ to be the minimal resolution of $Y'$, and take $Y_{\min}$ to be the minimal resolution of $Y$. Then the morphism $Y_{\min}\to Y'$ factors through $Y'_{\min}$.

\begin{claim}\label{claim:rho(Ymin)}
We have $\rho(Y_{\min}/Y)\leq \frac{8}{\delta}-1$.
\end{claim}

We grant Claim~\ref{claim:rho(Ymin)} for this moment and continue the proof of Theorem~\ref{thm lct1}. The proof of Claim~\ref{claim:rho(Ymin)} will be provided later. By Claim~\ref{claim:rho(Ymin)},
\begin{align*}
 \rho(Y'_{\min}/Y'){}&\leq \rho(Y_{\min}/Y')= \rho(Y_{\min}/Y)+\rho(Y)-1\\
 {}&\leq \rho(Y_{\min}/Y)+2\leq \frac{8}{\delta}+1.
\end{align*}
Here we used the fact that $\rho(Y)=\rho(S)+1\leq 3$.
Recall that $(Y', (1-\delta)E')$ is $\delta$-lc and $\delta<\frac{1}{6}$, then by Lemma~\ref{lem:multEC}, 
$\pi'^*E'$ is an effective $\mathbb{Q}$-divisor with all coefficients larger than $\frac{\delta^2}{8+2\delta}$. As $C$ is general, by the projection formula, $(E'\cdot C)> \frac{\delta^2}{8+2\delta}$,
which implies that $$\frac{e\delta^2}{8+2\delta}< (eE'\cdot C)=(-K_{Y'}\cdot C)\leq 2.$$ Hence $d\leq e\leq\frac{16+4\delta}{\delta^2}$. 

In summary, we always have $d \leq\frac{16+4\delta}{\delta^2}$. Therefore, by Eq.~\eqref{eq:t0>delta-2}, 
$$t_0\geq\frac{\epsilon-\delta}{d+\epsilon-1}\geq\frac{\delta^2(\epsilon-\delta)}{16+4\delta+\delta^2(\epsilon-1)}.
$$

\begin{proof}[Proof of Claim~\ref{claim:rho(Ymin)}]
Denote by $\pi: Y_{0}\to Y$ the minimal resolution of $Y$. Denote by $E_0$ the strict transform of $E$ on $Y_0$. 
Denote by $f_0=f\circ\pi: Y_0\to S$ the induced morphism.

Denote $P=f(E)\in S$. Then $Y\setminus E\simeq S\setminus \{P\}$ is smooth. 
Therefore $Y_0$ and $Y$ are isomorphic over $Y\setminus E$ and 
$$\text{Exc}(f_0)=f_0^{-1}(P)=\text{Exc}(\pi)\cup E_0.$$ 
Note that $f_0: Y_0\to S$ can be decomposed into successive blow-ups along smooth points and $\text{Exc}(\pi)$ does not contain any $(-1)$-curves, so $E_0$ is the unique $(-1)$-curve in $\text{Exc}(f_0)$. In other words, if we denote the last blow-up by $Y_0\to S_1$, then $E_0$ is the exceptional divisor over $S_1$. 

Denote by $\Gamma$ the weighted dual graph of $f_0$ and denote by $v_0$ the vertex corresponding to $E_0$. Then $\Gamma$ is a tree and $\Gamma\setminus\{v_0\}$ is the weighted dual graph of $\pi$. 
Since the weighted dual graph of $S_{1}\to S$ is also a tree, $\Gamma\setminus\{v_0\}$ has at most $2$ connected components (which implies that $Y$ has at most $2$ singular points). 

We claim that $\Gamma$ is a chain. Take $\Gamma'$ to be a connected component of $\Gamma\setminus\{v_0\}$, then it corresponds to exceptional curves over a singular point on $Y$.
Recall that $(Y, (1-\delta)E)$ is $\delta$-lc by Eq.~\eqref{eq:Y=f*S}.
As $\delta<\frac{1}{6}$,
$(Y, E)$ is lc by \cite[Corollary~6.0.9]{Prok01}.
Then the weighted dual graph $\Gamma'$ and its relation with $v_0$ are classified into $3$ cases as in \cite[Theorem~4.15]{KM}. We shall rule out \cite[Theorem~4.15(1)(2)]{KM}. In the case of \cite[Theorem~4.15(1)]{KM}, $\Gamma$ contains a loop, which is absurd; in the case of \cite[Theorem~4.15(2)]{KM}, $\Gamma$ contains a fork with 2 $(-2)$-curves on $2$ tails of the fork, so by contracting $(-1)$-curves in the graph successively, we will reach some model with $2$ $(-1)$-curves over $P\in S$, which is also absurd.
Hence we conclude that $\Gamma'$ is a chain connecting to $v_0$ by one edge at one end. Therefore $\Gamma$ is a chain.

If $\Gamma\setminus\{v_0\}$ is empty, then clearly $Y$ is smooth and $\rho(Y_{\min}/Y)=0$. So in the following we split the discussion into two cases, depending on the number of connected components of $\Gamma\setminus\{v_0\}$. 

\medskip

{\bf Case (a)}. $\Gamma\setminus\{v_0\}$ has $1$ connected component.

In this case, denote the $\pi$-exceptional curves by $E_1,\dots, E_N$ such that $E_0$ intersects $E_1$, where $N=\rho(Y_{0}/Y)$. Then all $E_{i}$ are $(-2)$-curves for $1\leq i\leq N$. 

Suppose that $Y_0\to S$ is decomposed into successive blow-ups at smooth points as 
$$
Y_0\to Y_1\to \dots \to Y_N\to Y_{N+1}=S,
$$
then $E_i$ is the strict transform of the exceptional divisor of $Y_{i}\to Y_{i+1}$. 
For each $i$, denote by $G_i$ the strict transform of $(1-t_0)B+t_0 D$ on $Y_i$ and denote by $P_i\in Y_i$ the blow-up center on $Y_i$. 

Write 
\begin{align}\label{eq:K+G}
{}&K_{Y_0}+G_0+(1-\delta)E_0+\sum_{i=1}^N b_i E_i=f_0^* (K_S+(1-t_0)B+t_0 D).
\end{align}
Then the coefficient of $E_0$ is computed from $G_1$ and $E_1$ by the formula
$$1-\delta=\mult_{P_1}G_1+b_1-1.$$
Since $b_1\leq 1-\delta$ as $( S, (1-t_0)B+t_0 D)$ is $\delta$-lc, we have $\mult_{P_1}G_1\geq 1$, hence $\mult_{P_i}G_i\geq \mult_{P_1}G_1\geq 1$ for $1\leq i\leq N+1$. Hence the intersection number $G_i^2 $ decreases at least by $1$ after each blow-up for $1\leq i\leq N+1$.
Therefore, 
\begin{align}\label{eq:N+1<G2-G2}
 N+1\leq G_{N+1}^2-G_0^2=K_S^2-G_0^2.
\end{align}

Since $Y_0$ is the minimal resolution of $Y$, we have $b_i\geq 0$ for $1\leq i\leq N$ in Eq.~\eqref{eq:K+G}. 
In particular, $({Y_0}, G_0+(1-\delta)E_0+\sum_{i=1}^N b_i E_i)$ is a $\delta$-lc pair such that
$$
K_{Y_0}+G_0+(1-\delta)E_0+\sum_{i=1}^N b_i E_i\equiv 0.
$$
Write $G_0=\sum_{k}c_k C_k$, where $C_k$ are distinct prime divisors, then $c_k\leq 1-\delta$. By Lemma~\ref{lem:C2>-2}, 
 $C_k^2\geq-\frac{2}{\delta}$. 
 
 If $S=\mathbb{F}_n$, then $\sum_{k}c_k\leq 4$ by \cite[Lemma~3.3]{Jiang21}. Hence
 \begin{align*}
 G_0^2=(\sum_k c_kC_k)^2\geq{}&(\sum_{k}c_{k}^2)\cdot(-\frac{2}{\delta})
 \geq(\sum_{k}c_{k})\cdot(1-\delta)\cdot(-\frac{2}{\delta})\\
 \geq{}&4(1-\delta)\cdot(-\frac{2}{\delta})
 =8-\frac{8}{\delta}.
 \end{align*}
 Combining with Eq.~\eqref{eq:N+1<G2-G2}, we have 
\begin{align*}
 N+1\leq{}&K_S^2-8+\frac{8}{\delta}= \frac{8}{\delta}. 
\end{align*}

If $S=\bP^2$, then we have $\sum_{k}c_k\leq 3$ and by the same argument we get \begin{align*}
 N+1\leq{}&K_S^2-G_0^2\leq 9-6+\frac{6}{\delta}\leq \frac{8}{\delta}. 
\end{align*}

\medskip

{\bf Case (b)}. $\Gamma\setminus\{v_0\}$ has $2$ connected components.

In this case, suppose that the 2 connected components are 2 chains $\Gamma_1, \Gamma_2$ consisting of vertices $v_1, \dots, v_p$ and $u_1, \dots, u_q$, corresponding to exceptional divisors $E_1, \dots, E_p$ and $F_1, \dots, F_q$ respectively, where $E_1$ and $F_1$ intersect $E_0$. Here $p+q=\rho(Y_{0}/Y)$.

Set 
$m_i=-E_{i}^2$ and $l_j=-F_{j}^2$ for $1\leq i \leq p$ and $1\leq j \leq q$.
Recall that $\Gamma$ is the weighted dual graph of a resolution over the smooth point $P\in S$, so after blowing down $E_0$, there is exactly one $(-1)$-curve among the strict transforms of $E_{1}$ and $F_{1}$. 
So without loss of generality, we may assume that $m_1\geq3$ and $l_1=2$.
Again by the fact that a contraction of a $(-1)$-curve in the graph induces another unique $(-1)$-curve, we know that $l_j=2$ for $1\leq j\leq m_1-2\leq q$, and $l_{m_1-1}\geq 3$ if $m_1-1\leq q$. 

Since $Y_0$ is the minimal resolution of $Y$, 
We may write $$K_{Y_0}+G'=f_0^*(K_S+(1-t_0)B+t_0 D)\equiv 0,$$
where $(Y_0, G')$ is a $\delta$-lc pair. 
By Lemma~\ref{lem:C2>-2}, we conclude that $m_1=-E_1^2\leq\frac{2}{\delta}$.

By Lemma~\ref{lem:LD}, we have 
\begin{align}
 {}&a(E_{i}, Y, (1-\delta)E)=\frac{\Delta(\Gamma_1\setminus \overline{v_{i}v_{p}})}{\Delta(\Gamma_1)}+\delta\cdot \frac{\Delta(\Gamma_1\setminus \overline{v_{1}v_{i}})}{\Delta(\Gamma_1)};\label{eq:aE}\\
 {}&a(F_{j}, Y, (1-\delta)E)=\frac{\Delta(\Gamma_2\setminus \overline{u_{j}u_{q}})}{\Delta(\Gamma_2)}+\delta\cdot \frac{\Delta(\Gamma_2\setminus \overline{u_{1}u_{j}})}{\Delta(\Gamma_2)}\label{eq:aF}
\end{align}
for $1\leq i \leq p$ and $1\leq j \leq q$.
 

To finish the proof, we claim that $p\leq \frac{1}{\delta}$ and $q\leq \frac{3}{\delta}-2$.
Recall that $(Y, (1-\delta)E)$ is $\delta$-lc.

First we show that $p\leq \frac{1}{\delta}$. We may assume that $p\geq2$.
By Eq.~\eqref{eq:aE},
$$a(E_{1}, Y, (1-\delta)E)=\frac{1}{\Delta(\Gamma_1)}+\delta\cdot \frac{\Delta(\Gamma_1\setminus\{v_{1}\})}{\Delta(\Gamma_1)}.$$
Therefore, $a(E_{1}, Y, (1-\delta)E)\geq \delta$ implies that
\begin{align*}
 \frac{1}{\delta}\geq {}& \Delta(\Gamma_1)- \Delta(\Gamma_1\setminus\{v_{1}\})
 >  \Delta(\Gamma_1\setminus\{v_{1}\}) \geq p.
\end{align*}
Here the second inequality is from Lemma~\ref{lem:chain}(6) with $i_0=1$ and the third is from Lemma~\ref{lem:chain}(4).

Next we show that $q\leq \frac{3}{\delta}-2$.
If $q=m_1-2$, then clearly $q\leq \frac{2}{\delta}-2$.
If $q>m_1-2$, then take $j_0=m_1-1$, we have $l_{j_0}\geq 3$. 
By Eq.~\eqref{eq:aF},
\begin{align*}
 a(F_{j_0}, Y, (1-\delta)E)={}&\frac{\Delta(\Gamma_2\setminus \overline{u_{j_0}u_{q}})}{\Delta(\Gamma_2)}+\delta\cdot \frac{\Delta(\Gamma_2\setminus \overline{u_{1}u_{j_0}})}{\Delta(\Gamma_2)}\\
 ={}& \frac{j_0}{\Delta(\Gamma_2)}+ \delta\cdot \frac{\Delta(\Gamma_2\setminus \overline{u_{1}u_{j_0}})}{\Delta(\Gamma_2)}.
\end{align*}
Therefore, $a(F_{j_0}, Y, (1-\delta)E)\geq \delta$ implies that
\begin{align*}
 \frac{j_0}{\delta}\geq {}& \Delta(\Gamma_2)- \Delta(\Gamma_2\setminus \overline{u_{1}u_{j_0}})
 > j_0\Delta(\Gamma_2\setminus \overline{u_{1}u_{j_0}})
 \geq  j_0(q-j_0+1).
\end{align*}
Here the second inequality is from Lemma~\ref{lem:chain}(6) with $i_0=j_0$ and the third is from Lemma~\ref{lem:chain}(4). Recall that $m_1=j_0+1$, so $q\leq\frac{1}{\delta}+m_1-2\leq\frac{3}{\delta}-2$.

In summary, $p+q\leq \frac{4}{\delta}-2$.
\end{proof}

\section{Upper bound of anti-canonical volumes}\label{sec 4}

In this section, we prove the main theorem.

\subsection{A reduction step}
The following proposition is a refinement of
\cite[Theorem~4.1]{Jiang21} by the idea of \cite[Proposition~4.1]{JZ21}.

\begin{prop}\label{prop MFS}
Fix $0<\epsilon<1$. Let $X$ be a $3$-fold of $\epsilon$-Fano type. 
Then $X$ is birational to a normal projective $3$-fold $W$ satisfying the following:
\begin{enumerate}
 \item $W$ is $\mathbb{Q}$-factorial terminal;
 \item $\Vol(X, -K_{X})\leq \Vol(W, -K_{W})$;
 \item $W$ is of $\epsilon$-Fano type; 
 \item there exists a projective morphism $f: W\to Z$ with connected fibers, such that one of the following conditions holds:
 \begin{enumerate}
 \item $Z$ is a point and $W$ is a $\mathbb{Q}$-factorial terminal Fano $3$-fold with $\rho(W)=1$;
 \item $Z= \mathbb{P}^1$;
 \item $Z$ is a del Pezzo surface with at worst Du Val singularities and $\rho(Z)=1$, and general fibers of $f$ are $\mathbb{P}^1$.
 \end{enumerate}
\end{enumerate}
\end{prop}

\begin{proof}
By \cite[Theorem~4.1]{Jiang21}, $X$ is birational to $W$ with a Mori fiber structure (see \cite[Definition~2.1]{Jiang21}), in particular, $W$ satisfies Properties (1)(2)(3).
So here we only need to explain how to get Property (4) by the proof of \cite[Proposition~4.1(5)]{JZ21}.

Denote by $W\to T$ the Mori fiber structure on $W$. 
 Note that $\dim T\in\{0,1,2\}$. By \cite[Theorem~1]{ZQ06}, $W$ is rationally connected, which implies that $T$ is also rationally connected.

If $\dim T=0$, then take $Z=T$ and $W$ is a $\mathbb{Q}$-factorial terminal Fano $3$-fold with $\rho(W)=1$. In this case, we get (a).

If $\dim T=1$, then $T\simeq \mathbb{P}^1$. In this case, we get (b).

If $\dim T=2$, then $T$ is a rational surface as it is rationally connected, and $T$ has at worst Du Val singularities by \cite[Theorem~1.2.7]{MP}. 
We can run a $K$-MMP on $T$ which ends up with a surface $T'$, such that either 
\begin{itemize}
 \item $T'$ is a del Pezzo surface with at worst Du Val singularities and $\rho(T')=1$, or 
 \item there is a morphism $T'\to \mathbb{P}^1$ with connected fibers. 
\end{itemize}
In the former case, take $Z=T'$ and take $f: W\to Z$ to be the induced morphism $W\to T\to T'$, then general fibers of $f$ are smooth rational curves as $-K_W$ is ample over $T$, then we get (c).
In the latter case, take $Z=\mathbb{P}^1$ and take $f: W\to Z$ to be the induced morphism $W\to T\to T'\to \mathbb{P}^1$, then we get (b). 
\end{proof}

\subsection{Proof of Theorem~\ref{mainthm}}
 
According to Proposition~\ref{prop MFS}, we can split the discussion into $3$ cases. We essentially follow the proof in \cite{Jiang21}.

\begin{prop}[{cf. \cite[Corollary~6.3]{Jiang21}}]\label{prop:dpf}
Suppose that $0<\epsilon<\frac{1}{3}$.
Keep the setting as in Proposition~\ref{prop MFS}. Assume that case (b) holds. Then 
$$\Vol(W, -K_{W})< \frac{3200}{\epsilon^4}.$$
\end{prop}

\begin{proof}
By \cite[Theorem~6.1]{Jiang21}, 
$$\Vol(W, -K_{W})\leq \frac{6M(2, \epsilon)}{\mu(2, \epsilon)},$$
where $M(2, \epsilon)\leq \frac{2}{\epsilon}+4+\frac{2}{3}< \frac{4}{\epsilon}$ by \cite{Jiang13} or \cite[Corollary~4.5]{Jiang21} and $\mu(2, \epsilon)>\frac{3\epsilon^3}{400}$ by Corollary~\ref{cor:genambro}.
\end{proof}

\begin{prop}[{cf. \cite[Theorem~6.6]{Jiang21}}]\label{prop:conicb}
Keep the setting as in Proposition~\ref{prop MFS}. Assume that case (c) holds. Then 
$$\Vol(W, -K_{W})\leq\frac{1152}{\epsilon^2}.$$
\end{prop}
\begin{proof}

By the classification of del Pezzo surfaces with at worst Du Val singularities and Picard rank $1$ (see \cite[Theorem~8.3.2]{Dol12}), there exists a base point free linear system $\mathcal{H}$ on $Z$ which defines a generically finite map such that $\mathcal{H}^2\leq 6$ (see also \cite[Proposition~4.3]{JZ21}).
Take a general element $H\in \mathcal{H}$ and denote $G=f^{-1}(H)$. 
Consider the self intersection number $d=H^2$.

By \cite[Lemma~6.5]{Jiang21}, we have $\Vol(G, -K_{W}|_G)\leq \frac{8(d+2)}{\epsilon}$. Then by \cite[Theorem~6.6]{Jiang21}, $\Vol(W, -K_{W}) \leq \frac{144(d+2)}{\epsilon^2}$. This proves the proposition as $d\leq 6.$

Here we remark that in \cite[Lemma~6.5, Theorem~6.6]{Jiang21}, the assumptions are 
\begin{itemize}
 \item $W\to Z$ is a Mori fiber space and $\dim Z=2$;
 \item $\mathcal{H}$ is very ample.
\end{itemize}
But those assumptions can be slightly weaken as in our setting without any other changes to the proofs:
\begin{itemize}
 \item general fibers of $W\to Z$ are $\mathbb{P}^1$;
 \item $\mathcal{H}$ is base point free and defines a generically finite map.
 \end{itemize} 
\end{proof}

\begin{proof}[Proof of Theorem~\ref{mainthm}]
By Proposition~\ref{prop MFS}, it suffices to bound $\Vol(W, -K_{W})$. If case (a) holds, then $\Vol(W, -K_{W})=(-K_W)^3\leq 64$ by \cite{Prok07}. If case (b) or (c) holds, then the conclusion follows from Propositions~\ref{prop:dpf} and \ref{prop:conicb}.
\end{proof}

\section*{Acknowledgments} 
The authors are grateful to Caucher Birkar for discussions and suggestions. The second author would like to thank her mentor, Professor Caucher Birkar, for his support and encouragement.
This work was supported by National Natural Science Foundation of China for Innovative Research Groups (Grant No. 12121001) and National Key Research and Development Program of China (Grant No.~2020YFA0713200). The first author is a member of LMNS, Fudan University.

\end{document}